\documentclass[11pt]{amsart}

\usepackage{amsmath,amsfonts,amssymb,amsthm}
\usepackage{latexsym,bm,graphicx}
\usepackage{mathrsfs}
\usepackage{color}

\title[Liouville type theorem for elliptic inequality]{A note on Liouville type theorem of elliptic inequality $\Delta u+u^\sigma\ls 0$ on Riemannian manifolds}
\author{Hui-Chun Zhang}
\address{Department of Mathematics\\  Sun Yat-sen University\\ Guangzhou 510275\\ E-mail address: zhanghc3@mail.sysu.edu.cn}

\newtheorem{thm}{Theorem}[section]

\newtheorem{lem}[thm]{Lemma}

\theoremstyle{definition}
\theoremstyle{remark}

\newtheorem{rem}[thm]{Remark}

\numberwithin{equation}{section}

\newcommand{\ls}{\leqslant}
\newcommand{\gs}{\geqslant}

\newcommand{\ip}[2]{\left<{#1},{#2}\right>}

\newcommand{\R}{\mathbb{R}}

\begin{document}

\begin{abstract}
Let $\sigma>1$ and let $M$ be a complete Riemannian manifold. In a very recent work \cite{gs13}, Grigor$^{\prime}$yan and Sun proved that a Liouville type theorem for nonnegative solutions of elliptic inequality
$$(*)\quad\qquad\qquad\qquad \Delta u(x)+u^\sigma(x)\ls0,\qquad x\in M.\quad\qquad \qquad\qquad $$
via a pointwise condition of volume growth of geodesic balls. In this note, we improve their result to that an \emph{integral condition} on volume growth implies the same uniqueness of ($*$). It is inspired by the well-known Varopoulos-Grigor$^{\prime}$yan's criterion for parabolicity of $M$.
\end{abstract}
\maketitle

\section{Introduction}

Let $\sigma>1$ and let $M$ be a complete noncompact Riemannian manifold without boundary.
Consider the semilinear elliptic inequality
\begin{equation}\label{eq1.1}
\Delta u(x)+u^\sigma(x)\ls0, \quad x\in M,
\end{equation}
 where $\Delta$ is the Laplace-Beltrami opertor on $M$.
 A function $u\in W^{1,2}_{\rm loc}(M)$  is called a \emph{weak solution} of the inequality \eqref{eq1.1} if
$$-\int_M\ip{\nabla u}{\nabla \psi}d\mu+\int_Mu^\sigma\psi d\mu\ls0$$
holds for any nonnegative function $\psi\in W^{1,2}(M)$ with compact support.

In Euclidean setting, i.e. $M=\R^n$, it has a long history to study the uniqueness of nonnegative solutions for \eqref{eq1.1} (or more general elliptic inequations and equalities). There are many beautiful results have been obtained in this subject. We refer the readers to, for instance, \cite{agq13,bp01,cgs94,cmp09,p09,ser72,sz02} and references therein for them. Many of these results are based on comparison principle and careful choices of test functions for \eqref{eq1.1}. To use this method on a manifold $M$, one have to estimate the second order derivative of distance functions, which needs some assumptions on curvature of $M$.

 Surprisingly, in recent works  Grigor$^{\prime}$yan-Kondratiev \cite{gk10} and Grigor$^{\prime}$yan-Sun \cite{gs13} proved a \emph{curvature-free} Liouville type theorem for nonnegative weak solution of \eqref{eq1.1} in terms of  volume growth of geodesic balls in $M$ as follows.

 \begin{thm}[Grigor$^{\prime}$yan-Sun \cite{gs13}]
Let $M$ be a complete  Riemannian manifold without boundary. Fix a point $x_0\in M$ and set
$V(r):=\mu \big(B(x_0,r)\big)$
the volume of geodesic ball of radius $r$ centered at $x_0$.

 Assume that, for some  $C>0$, the inequality
\begin{equation}\label{eq1.2}
V(r)\ls C r^\frac{2\sigma}{\sigma-1}(\ln r)^\frac{1}{\sigma-1}
\end{equation}
holds for all large enough $r$.
Then any nonnegative weak solution of \eqref{eq1.1} is identically equal to 0.
\end{thm}
\noindent They also showed that the exponents $\frac{2\sigma}{\sigma-1}$ and $\frac{1}{\sigma-1}$ are sharp.

On the other hand, let us recall that
a manifold $M$ is said to be \emph{parabolic} if a Liouville type theorem holds for nonnegative solution of inequality
$$\Delta u(x)\ls0,\qquad x\in M, $$ i.e.,
any nonnegative weak solution of  $\Delta u\ls0 $ on $M$ must be constant.
 Cheng and Yau \cite{cy75} proved that $V(r)\ls C r^2$, for some $C>0$, is a sufficient condition for parabolicity of $M$. Nowdays, a well-known sharp sufficient condition for parabolicity is the following \emph{integral} condition, which was proved independly by Varopoulos \cite{v81} and Grigor$^\prime$yan \cite{g83,g85}:
$$\int^\infty \frac{r}{V(r)}dr=\infty.$$

Inspired  by Varopoulos-Grigor$^\prime$yan's condition for  the parabolicity of $M$, we ask a natural question: what is a sufficient condition for Liouville type theorem of inequlity \eqref{eq1.1} via an \emph{integral} estimate of $V(r)$? Of course, such a condition should cover the above pointwise condition \eqref{eq1.2}.

In this remark, we solve this question. Our main result states as follows:
\begin{thm}\label{main}
Let $M$ be a complete  Riemannian manifold without boundary. Assume that
\begin{equation}\label{eq1.3}
\liminf_{t\to0^+}t^{\frac{\sigma}{\sigma-1}}\int_1^\infty\frac{V(r)}{r^{\frac{3\sigma-1}{\sigma-1}+t}}dr<\infty.
\end{equation}
Then any nonnegative weak solution of \eqref{eq1.1} is identically equal to 0.
\end{thm}
\begin{rem}Condition \eqref{eq1.2} in Theorem 1.1 implies the condition \eqref{eq1.3}. In fact,
$$\eqref{eq1.2}\Longrightarrow t^{\frac{\sigma}{\sigma-1}}\int_1^\infty \frac{V(r)dr}{r^{\frac{3\sigma-1}{\sigma-1}+t}}\ls t^{\frac{\sigma}{\sigma-1}}\int_1^\infty \frac{(\ln r)^{\frac{1}{\sigma-1}}dr}{r^{1+t}}=\Gamma(\frac{\sigma}{\sigma-1}),
$$
where $\Gamma(\cdot)$ is Gamma function.
\end{rem}

\section{Proof of Theorem \ref{main}}
\begin{proof}[Proof of Theorem \ref{main}]
Let $u\in W^{1,2}_{\rm loc}(M)$ be a nontrivial nonnegative solution to the inequality \eqref{eq1.1}.

The proof of Theorem 1.1 in \cite{gs13} contains two main parts. Firstly, the authors derived a useful priori estimate
in terms of a test function and positive parameters (which will be recalled in Lemma 2.1 below).
Secondly, they chose specific test functions to conclude $\int_Mu^\sigma d\mu=0.$
 Our proof of Theorem 1.2 is basically along the same line in \cite{gs13}. The different from Grigor$^{\prime}$yan-Sin's proof will appear in the second part. We will choose a variation of their test functions
 to conclude  $\int_Mu^\sigma d\mu=0.$

Firstly, let us recall the useful priori estimate given in \cite{gs13}. We summarize it as the following lemma:
\begin{lem}[Grigor$^{\prime}$yan-Sun, \cite{gs13}] Set $s=8\sigma/(\sigma-1)$.
Then there exists a constant $C_0>0$ such that the following property holds:

For any
 $$t\in \big(0,\min\{1,\frac{\sigma-1}{2}\}\big),$$
  any nonempty compact set $K\subset M$, and any Lipschitz function $\phi$ on $M$ with conpact support such that $0\ls \phi\ls1$ on $M$ and $\phi\equiv1$ in a neighborhood of $K$, we have
\begin{equation}\label{eq2.1}
\int_M \phi^s u^\sigma d\mu\ls C_0\Big(\int_{M\backslash K}\phi^su^\sigma d\mu \Big)^{\frac{t+1}{2\sigma}}\cdot J(t,\phi)
\end{equation}
and
\begin{equation}\label{eq2.2}
\Big(\int_M \phi^s u^\sigma d\mu\Big)^{1-\frac{t+1}{2\sigma}}\ls C_0\cdot J(t,\phi),
\end{equation}
where
$$J(t,\phi):=t^{-\frac{1}{2}-\frac{\sigma}{2(\sigma-1)}}\Big(\int_M|\nabla\phi|^{2\frac{\sigma-t}{\sigma-1}}d\mu\Big)^{\frac{1}{2}}
\cdot\Big(\int_M|\nabla\phi|^{\frac{2\sigma}{\sigma-t-1}}d\mu\Big)^{\frac{\sigma-t-1}{2\sigma}}.$$
\end{lem}
\begin{proof}Inequality \eqref{eq2.1} is Eq.(2.10) in \cite{gs13}, and inequality \eqref{eq2.2} is Eq.(2.11) in \cite{gs13}.
\end{proof}

In the following, we will consider a family of specific test functions $\phi_n$, which are modifications from original structures in \cite{gs13}.

Fix any $t\in \big(0,\min\{1,\frac{\sigma-1}{2}\}\big)$. We set $R=R(t):=\exp(1/t)$.
 We consider the function
\begin{equation*}
\phi_t(x)=
\begin{cases}1, &r(x)<R,\\
\Big(\frac{r(x)}{R}\Big)^{-t},& r(x)\gs R,
\end{cases}
\end{equation*}
 and a family of functions, for any $n=1,2,3,\cdots,$
\begin{equation*}
\xi_{t,n}(x)=
\begin{cases}1, &0\ls r(x)\ls 2^{n}R,\\
 2-\frac{r(x)}{2^{n}R},& 2^{n}R\ls r(x)\ls 2^{n+1}R,\\
 0,& r(x)\gs 2^{n+1}R.
\end{cases}
\end{equation*}
Consider the functions
\begin{equation}\label{eq2.3}
\phi_{t,n}(x):=\phi_t(x)\cdot \xi_{t,n}(x).
\end{equation}
Then, for each $n=1,2,\cdots$, function $\phi_{t,n}(x)$ is Lipschitz continuous on $M$ and has compact support, and $\phi_{t,n}\equiv1$ on $B_{R(t)}:=B(x_0,R(t))$.

\noindent\textbf{Claim:} \emph{There exists a constant $C_1>0$ such that, for any $t\in\! \big(0,\min\{1,\frac{\sigma-1}{2}\}\big)$ with
$$A(t):=\int_1^\infty \frac{V(r)}{r^{\frac{3\sigma-1}{\sigma-1}+t}}dr<\infty,$$
we have
\begin{equation}\label{eq2.4}
\limsup_{n\to\infty}[J(t,\phi_{t,n})]^{\frac{2\sigma}{2\sigma-t-1}}\ls C_1\cdot t^{\frac{\sigma}{\sigma-1}}\cdot A(t).
\end{equation}
}

\begin{proof}[Proof of Claim:]
In the proof, the parameter $t$ is fixed. To simplify the notations, we denote by
$$\phi:=\phi_t,\quad \xi_n:=\xi_{t,n} \quad {\rm and}\quad \phi_n:=\phi_{t,n}.$$

Notice that
$$
\nabla \phi_{n}=\xi_{n}\cdot\nabla\phi  +\phi \cdot\nabla\xi_{n}.
$$
We have
$$
|\nabla \phi_{n}|\ls \xi_{n}\cdot|\nabla\phi | +\phi \cdot|\nabla\xi_{n}|;
$$
and, by the inequality $(A+B)^a\ls 2^{a-1}(A^a+B^a)$ for all $A,B>0$ and $a\gs1$,
$$
|\nabla \phi_{n}|^a\ls 2^{\frac{4\sigma}{\sigma-1}-1}\big[ \xi_{n}^a\cdot|\nabla\phi |^a +\phi ^a\cdot|\nabla\xi_{n}|^a\big]
$$
for any $a\in [1,\frac{4\sigma}{\sigma-1}]$. In the following, we denote by
 $$\sigma_0:=\frac{4\sigma}{\sigma-1}.$$

Similar as in \cite{gs13}, we need to estimate the integral $\int_M|\nabla \phi_{n}|^ad\mu.$
 For any $a\in [1,\sigma_0]$, we have
\begin{equation}\label{eq2.5}
\begin{split}
\int_M\!|\nabla \phi_{n}|^ad\mu&\ls  2^{\sigma_0-1}\cdot\Big(\int_{M\backslash B_{R }}\!\!|\nabla \phi |^a d\mu+\!\int_{B_{2^{n+1}R}\backslash B_{2^{n}R }}\!\!\phi ^a|\nabla \xi_{n}|^a d\mu\Big)\\
&:= 2^{\sigma_0-1}\cdot \big(I(a)+ I\!I(a,n)\big),
\end{split}
\end{equation}
where $B_R:=B(x_0,R)$, and we have used that $\nabla \phi =0$ in $B_{R}$ and that $|\nabla \xi_{n}|$ supported in $\overline{B_{2^{n+1}R }}\backslash B_{2^{n}R }$.

Before we estimate the above integrals $I(a)$ and $I\!I(a,n)$, we need the following simple (but important) observation:

 \emph{ If the parameter $a\in [1,\sigma_0]$ satisfies
\begin{equation}\label{eq2.6}
a(t +1)\gs t +\frac{2\sigma}{\sigma-1}.
\end{equation}
Then we have
\begin{equation}\label{eq2.7}
\sum_{n=1}^\infty\frac{ V(2^nR )}{\big(2^{n-1}R \big)^{a(t +1)}}
\ls 2\cdot 16^{\sigma_0}\cdot A(t):=C_2\cdot A(t).
\end{equation}
In particular, it implies that
\begin{equation}\label{eq2.8}
\lim_{n\to\infty}\frac{ V(2^nR )}{\big(2^{n-1}R \big)^{a(t +1)}}=0.
\end{equation}}

Indeed, we calculate directly to conclude
\begin{equation}\label{eq2.9}
\begin{split}
\sum_{n=1}^\infty&\frac{ V(2^nR )}{\big(2^{n-1}R \big)^{a(t +1)}}\\
&\quad =4^{a(t +1)}\cdot 2\cdot\sum_{n=1}^\infty\frac{V(2^nR )}{\big(2^{n+1}R \big)^{a(t +1)}}
\cdot \frac{2^{n+1}R -2^nR }{2^{n+1}R }\\
&\quad\ls 4^{a(t +1)}\cdot 2
\cdot\sum_{n=1}^\infty\int_{2^nR }^{2^{n+1}R }\frac{V(r)dr}{r^{a(t +1)+1}}\\
&\quad\ls 2\cdot 16^{\sigma_0}\cdot \int_{1}^{\infty}\frac{V(r)dr}{r^{a(t +1)+1}},
\end{split}
\end{equation}
we we have used that $t<1$, $a\ls \sigma_0$ and that $R=\exp(1/t)>1$. Combining with \eqref{eq2.6} and \eqref{eq2.9}, we can obtain
\begin{equation*}
\sum_{n=1}^\infty\frac{ V(2^nR )}{\big(2^{n-1}R \big)^{a(t +1)}}
\ls 2\cdot 16^{\sigma_0}\int_{1}^{\infty}\frac{V(r)dr}{r^{t+\frac{2\sigma}{\sigma-1}+1}}=2\cdot 16^{\sigma_0}\cdot A(t).
\end{equation*}
Ths is the desired estimate (2.7).

Now let us estimate $I(a)$. Assume that the parameter $a$ satisfies \eqref{eq2.6}, we have
\begin{equation}\label{eq2.10}
\begin{split}
I(a)&=\int_{M\backslash B_{R }}\!\!\!|\nabla \phi |^a d\mu\ls \int_{M\backslash B_{R }}\Big[\frac{t }{R }\cdot\Big(\frac{r}{R }\Big)^{-t -1}\Big]^a d\mu\\
&=e^a\cdot t^a \int_{M\backslash B_{R }}\!\frac{1}{r^{a(t +1)}} d\mu\qquad\quad ({\rm since}\ R ^{t }=e)\\
&=e^a\cdot t^a \cdot\sum_{n=1}^\infty\int_{B_{2^nR }\backslash B_{2^{n-1}R }}\!\frac{1}{r^{a(t +1)}} d\mu\\
&\ls e^a\cdot t^a \cdot\sum_{n=1}^\infty\frac{ V(2^nR )}{\big(2^{n-1}R \big)^{a(t +1)}}\\
&\ls e^{\sigma_0}\cdot C_2\cdot t^aA(t) \qquad \big({\rm by}\ \ a\ls \sigma_0\ {\rm and}\ \ \eqref{eq2.7} \big).
\end{split}
\end{equation}
Let us estimate $I\!I(a,n)$. Assume that the parameter $a$ satisfies \eqref{eq2.6}, we have
\begin{equation}\label{eq2.11}
\begin{split}
I\!I(a,n)&=\int_{B_{2^{n+1}R }\backslash B_{2^{n}R }}\!\!\phi ^a|\nabla \xi_{n}|^a d\mu\\
&\ls \Big(\frac{2^nR}{R})^{-at}\Big(\frac{1}{2^nR}\Big)^a\cdot V(2^{n+1}R)\\
&=R^{at}\cdot\frac{V(2^{n+1}R)}{(2^nR)^{a(t+1))}}\overset{R^t=e}{=}e^a\cdot\frac{V(2^{n+1}R)}{(2^nR)^{a(t+1))}}.
\end{split}
\end{equation}
Combining with \eqref{eq2.8}, \eqref{eq2.11} and that  $a\ls \sigma_0$, we have
\begin{equation}\label{eq2.12}
\lim_{n\to \infty}I\!I(a,n)=0.
\end{equation}
Therefore, according to \eqref{eq2.5},\eqref{eq2.10} and \eqref{eq2.12}, we obtain, for any $a\in [1,\sigma_0]$ satisfying \eqref{eq2.6},
\begin{equation}\label{eq2.13}
\limsup_{n\to\infty}\int_M|\nabla \phi_n|^{a}d\mu\ls  2^{\sigma_0-1}\cdot  e^{\sigma_0}\cdot C_2\cdot t^aA(t):=C_3\cdot t^aA(t).
\end{equation}

We take
$$a_1=2\frac{\sigma-t}{\sigma-1} \qquad {\rm and}\quad  a_2=\frac{2\sigma}{\sigma-t-1}.$$
Then it is easy to check that $a_1,a_2$ satisfy \eqref{eq2.6}. Indeed,
\begin{equation*}
a_1(t+1)=\frac{2\sigma}{\sigma-1}+2t-\frac{2t^2}{\sigma-1}
\gs \frac{2\sigma}{\sigma-1}+t \qquad ({\rm since }\ \ t\ls  \frac{\sigma-1}{2}  )
\end{equation*}
and
\begin{equation*}
a_2(t+1)=\frac{2\sigma}{\sigma-1}\cdot\frac{\sigma-1}{\sigma-t-1}\cdot (t+1)\gs\frac{2\sigma}{\sigma-1}\cdot (t+1)\gs  \frac{2\sigma}{\sigma-1}+t.
\end{equation*}
Now, by using
$$J(t,\phi_n)=t^{-\frac{1}{2}-\frac{\sigma}{2(\sigma-1)}}\Big(\int_M|\nabla\phi_n|^{a_1}d\mu\Big)^{\frac{1}{2}}
\cdot\Big(\int_M|\nabla\phi_n|^{a_2}d\mu\Big)^{\frac{1}{a_2}}$$
and \eqref{eq2.13}, we can conclude that
\begin{equation*}
\begin{split}
\limsup_{n\to\infty}J(t,\phi_n)&\ls  t^{-\frac{1}{2}-\frac{\sigma}{2(\sigma-1)}}\cdot C^{\frac{1}{2}+\frac{1}{a_2}}_3\cdot t^{\frac{a_1}{2}+1}\cdot
 [A(t)]^{\frac{1}{2}+\frac{1}{a_2}} \\
&= C_3^{\frac{2\sigma-t-1}{2\sigma}}\cdot t^{\frac{1}{2}+\frac{\sigma}{2(\sigma-1)}-\frac{t}{\sigma-1}}\cdot[A(t)]^{\frac{2\sigma-t-1}{2\sigma}}
\end{split}
\end{equation*}
Then
\begin{equation}\label{eq2.14}
\begin{split}
\limsup_{n\to\infty}[J(t,\phi_n)]^{\frac{2\sigma}{2\sigma-t-1}}&\ls C_3\cdot
t^{(\frac{1}{2}+\frac{\sigma}{2(\sigma-1)}-\frac{t}{\sigma-1})\cdot\frac{2\sigma}{2\sigma-t-1}}\cdot A(t)\\
&=C_3\cdot t^{\frac{\sigma}{\sigma-1}\cdot(1-\frac{t}{2\sigma-t-1})}\cdot A(t).
\end{split}
\end{equation}
Noticing that
$$\lim_{t\to0^+} t^{-\frac{\sigma}{\sigma-1}\cdot \frac{t}{2\sigma-t-1}}=1,$$
we have that the function $t\mapsto t^{-\frac{\sigma}{\sigma-1}\cdot \frac{t}{2\sigma-t-1}}$ is bounded on $(0,1)$ uniformly.

Set the constant
$$C_1:=C_3\cdot \sup_{0<t<1}t^{-\frac{\sigma}{\sigma-1}\cdot \frac{t}{2\sigma-t-1}}.$$
Then the desired estimate \eqref{eq2.4} follows from \eqref{eq2.14}, and hence the proof of \textbf{Claim} is completed.
\end{proof}

Now let us continue the proof of Theorem 1.2.

According to \eqref{eq1.3},  there is a sequence of numbers $\{t_\alpha\}_{\alpha=1}^\infty$, going to 0, such that
\begin{equation}\label{eq2.15}
t_\alpha^{\frac{\sigma}{\sigma-1}}\cdot A(t_\alpha)=t_\alpha^{\frac{\sigma}{\sigma-1}}\int_1^\infty\frac{V(r)}{r^{\frac{3\sigma-1}{\sigma-1}+t_\alpha}}dr\ls C_4,\qquad \forall \ \alpha=1,2,\cdots
\end{equation}
for some constant $C_4$, independent of $\alpha.$ Without loss the generality,  we can also assume that  $t_{\alpha}\in(0,\min\{1,\frac{\sigma-1}{2}\}),$ for all $\alpha=1,2,3,\cdots.$

By using the above \textbf{Claim}, we have, for each $\alpha=1,2,\cdots$,
\begin{equation}\label{eq2.16}
\limsup_{n\to\infty}\ J(t_\alpha,\phi_{t_\alpha,n})\!\ls\! (C_1\cdot C_4)^{\frac{2\sigma-t_\alpha-1}{2\sigma}}\!\ls\!\max\{(C_1 C_4)^{\frac{2\sigma-1}{2\sigma}},1\}\! :=C_5.
\end{equation}

In the following is similar as in \cite{gs13}. We want to show $u\in L^\sigma(M)$, and moreover $\int_Mu^\sigma d\mu=0.$  Fix arbitrary a nonempty compact set $K\subset M$.

 Notice that $R(t_\alpha)=\exp(1/t_\alpha)\to\infty$ as $\alpha\to\infty$.
So, we have
$$K\subset B_{R(t_\alpha)}$$
for all large enough $\alpha$. Hence, for any sufficient large $\alpha$, $\phi_{t_\alpha,n}\equiv1$ on $K$ holds for any $n=1,2,\cdots.$.
For such $\alpha$, we can apply Lemma 2.1 to $t_\alpha,$ $K$ and function $\phi_{t_\alpha,n}$; and we conclude that
\begin{equation}\label{eq2.17}
 \int_{K}u^\sigma d\mu\!\ls \int_{M}\phi_{t_\alpha,n}^su^\sigma d\mu\!\ls C_0 \Big(\int_{M\backslash K}\!\phi_{t_\alpha,n}^su^\sigma d\mu\Big)^{\frac{t_\alpha+1}{2\sigma}}\cdot  J(t_\alpha,\phi_{t_\alpha,n})
 \end{equation}
and
\begin{equation}\label{eq2.18}
\int_{K}u^\sigma d\mu\ls \int_{M}\phi_{t_\alpha,n}^su^\sigma d\mu \ls\Big(C_0\cdot J(t_\alpha,\phi_{t_\alpha,n})\Big)^{\frac{2\sigma}{2\sigma-t_\alpha-1}},
\end{equation}
for all $n=1,2,\cdots$, where we have used that $\phi_{t_\alpha,n}\equiv1$ on $K$.

By combining \eqref{eq2.16} and \eqref{eq2.18}, we obtain
\begin{equation*}
\int_{K}u^\sigma d\mu \ls\Big(C_0\cdot C_5\Big)^{\frac{2\sigma}{2\sigma-t_\alpha-1}}
\end{equation*}
for all large enough $\alpha$. Letting $\alpha\to\infty$, we have
\begin{equation}\label{eq2.19}
\int_{K}u^\sigma d\mu\ls\Big(C_0\cdot C_5\Big)^{\frac{2\sigma}{2\sigma-1}}:=C_6.
\end{equation}
By combining with\eqref{eq2.17},\eqref{eq2.19},\eqref{eq2.16} and that $\phi_{t_\alpha,n}\ls1$ on $M$, we have
\begin{equation*}
 \int_{K}u^\sigma d\mu\! \ls C_0 \Big(\int_{M\backslash K}\!u^\sigma d\mu\Big)^{\frac{t_\alpha+1}{2\sigma}}\cdot C_5,
 \end{equation*}
for all large enough $\alpha$. Letting $\alpha\to\infty$, we have
\begin{equation}\label{eq2.20}
 \int_{K}u^\sigma d\mu\! \ls C_0 \cdot C_5\cdot \Big(\int_{M\backslash K}\!u^\sigma d\mu\Big)^{\frac{1}{2\sigma}}.
 \end{equation}
By using the arbitrariness of $K$, we can take $K=\overline{B_r}$ for any $r>0$.
Combining with \eqref{eq2.19} and \eqref{eq2.20} and letting $r\to\infty$, we have
$$\int_Mu^\sigma d\mu=0,$$
which implies $u\equiv0$ on $M$, and the proof of Theorem \ref{main} is completed.
\end{proof}

\textbf{Acknowledgements.} We would like to thank Dr. Yuhua Sun for
his interesting in the paper. The author is partially supported by Guangdong Natural Science Foundation S2012040007550 and by China Postdoctoral Science Foundation  2012T50736, 2012M521639.


\begin{thebibliography}{99}

\bibitem{agq13} S. Alarc\'on; J. Garci\'a-Meli\'an \& A. Quaas, \emph{Nonexistence of positive supersolutions to some nonlinear elliptic porblems,} J. Math. Pures Appl. 99(2013) 618--634.

\bibitem{bp01} Biduat-Veron, M.-F., \& Pohozaev, S., \emph{Nonexistence results and estimates for some nonlinear elliptic problems}, J. Anal. Math., 84 (2001), 1--49.

\bibitem{cgs94} Caffarelli, L.; Garofalo, N. \& Segala, F., \emph{A gradient bound for entire solutions of quasi-linear equatios and its consequences}, Comm. Pure Appl. Math., 47 (1994), 1457--1473.

\bibitem{cmp09} Caristi, G.; Mitidieri, E.\& Pohozaev, S. I., \emph{ Some Liouville theorems for quasilinear elliptic inequalities}, Doklady Math. 79 (2009), no. 1, 118--124.

\bibitem{cy75} S. Y. Cheng \& S. T. Yau, \emph{Differential equations on Riemannian manifolds and their geometric applications}, Comm. Pure Appl. Math., 28(1975), no. 3, 333--354.

\bibitem{gidas81} Gidas, B.\& Spruck, J., \emph{Global and local behavior of positive solutions of nonlinear elliptic equations}, Comm. Pure Appl. Math. 34 (1981), no. 4, 525--598.

\bibitem{g83} A. Grigor$^{\prime}$yan, \emph{On the existence of a Green function on a manifold} (in Russian), Uspekhi Math. Nauk 38(1)(1983) 161--162, Engl. transl.: Russian Math. Surveys 38(1)(1983) 190--191.

 \bibitem{g85} A. Grigor$^{\prime}$yan, \emph{On the existence of positive fundamental solution of the Laplace equation on Riemannian manifolds} (in Russian), Maat. Sb. 128(3)(1985) 354--363, Engl. transl.:  Math. USSR Sb. 56 (1987) 349--358.
     
 \bibitem{gk10}   A. Grigor$^{\prime}$yan \&  V. A. Kondratiev, \emph{On the existence of positive solutions of semilinear elliptic
inequalities on Riemannian manifolds}. Around the research of Vladimir Maz¡¯ya II, 203--218.
International Mathematical Series (New York), 12. Springer, New York, 2010.

\bibitem{gs13} A. Grigor$^{\prime}$yan \& Y. Sun, \emph{On Nonnegative Solutions of the Inequality $\Delta u+u^\sigma\ls0$ on Riemnannian Manifolds}, to appear in Comm. Pure Appl. Math., (2013).


\bibitem{p09} Pohozaev, S. I., \emph{Critical nonlinearities in partial differential equations}, Milan J. Math. 77 (2009), 127--150.

\bibitem{ser72}J. Serrin, \emph{Entire solutions of nonlinear Poisson equations}, Proc. London Math. Soc. (3), 24 (1972), 348--366.

\bibitem{sz02} J. Serrin \& H. Zou, \emph{Cauchy-Liouville and universal boundedness theorems for quasilinear elliptic equations and inequalities}, Acta Math., 189(2002), 79--142.

\bibitem{v81} N. Varopoulos, \emph{The Poisson kernel on positively curved manifolds}, J. Funct. Anal. 44(1981), 359--380.



\end{thebibliography}
\end{document}